\documentclass{article}
\usepackage{amsmath,amsthm,amsfonts,amssymb,amscd}
\usepackage{enumerate}

\newtheorem{thm}{Theorem}[section]
\newtheorem{prop}[thm]{Proposition}
\newtheorem{lem}[thm]{Lemma}

\newtheorem{prob}[thm]{Problem}

\newtheorem{ex}[thm]{Example}
\theoremstyle{definition}

\newtheorem{rem}[thm]{Remark}

\newcommand{\Z}{\mathbb{Z}}

\newcommand{\R}{\mathbb{R}}
\newcommand{\C}{\mathbb{C}}

\renewcommand{\a}{\mathfrak{a}}
\newcommand{\g}{\mathfrak{g}}
\newcommand{\h}{\mathfrak{h}}
\renewcommand{\j}{\mathfrak{j}}
\renewcommand{\k}{\mathfrak{k}}
\newcommand{\n}{\mathfrak{n}}
\newcommand{\z}{\mathfrak{z}}

\renewcommand{\L}{\mathcal{L}}

\newcommand{\w}{\wedge}
\newcommand{\bs}{\backslash}
\newcommand{\tr}{\operatorname{tr}}
\newcommand{\ad}{\operatorname{ad}}

\newcommand{\Stab}{\operatorname{Stab}}
\newcommand{\stab}{\mathfrak{stab}}

\renewcommand{\leq}{\leqslant}

\newcommand{\bysame}{------}

\begin{document}

\title{Homogeneous spaces of nonreductive type locally modelling no compact manifold}
\author{Yosuke Morita}
\date{}

\maketitle

\begin{abstract}
We give necessary conditions for the existence of a compact manifold locally modelled on a given homogeneous space, which generalize some earlier results, 
in terms of relative Lie algebra cohomology. 
Applications include both reductive and nonreductive cases. 
For example, we prove that there does not exist a compact manifold locally modelled on a positive dimensional coadjoint orbit of a real linear solvable algebraic group. 
\end{abstract}

\section{Introduction}

Let $G/H$ be a homogeneous space. 
A manifold is called locally modelled on $G/H$ 
if it is covered by open sets that are diffeomorphic to open sets of $G/H$ 
and their coordinate changes are given by left translations by elements of $G$. 
A typical example is a double coset space $\Gamma \bs G/H$, 
where $\Gamma$ is a discrete subgroup of $G$ acting properly and freely on $G/H$. 
In this case $\Gamma$ is called a discontinuous group for $G/H$ and 
$\Gamma \bs G/H$ is called a Clifford--Klein form. 
A manifold locally modelled on a homogeneous space is a fundamental object of the study of ``geometry'' in the sense of Klein's Erlangen program. 
Thus, one of the central questions in geometry is to understand topological features of manifolds locally modelled on a given homogeneous space. 

In this paper, we study the following problem proposed by T. Kobayashi: 

\begin{prob}[\cite{Kob89}]\label{prob:cpt}
Which homogeneous space can locally model a compact manifold?
Which homogeneous space admits a compact Clifford--Klein form? 
\end{prob}

Various methods have been applied to study this problem 
(See surveys \cite{Kob96survey}, \cite{Lab96}, \cite{Kob-Yos05}, \cite{Con13} and references therein). 
One is a cohomological method, that is, 
to investigate ``locally invariant'' differential forms on a manifold locally modelled on a homogeneous space and their cohomology classes. 
This method was initiated by Kobayashi--Ono \cite{Kob-Ono90} and 
has been used and extended in \cite{Ben-Lab92} and \cite{Mor}. 
In this paper, we find that this method is useful even when $G$ is not reductive. 
Note that, for a nonreductive Lie group $G$, less is known about Problem~\ref{prob:cpt} 
in particular because we cannot use the properness criterion of Benoist \cite{Ben96} and Kobayashi \cite{Kob96proper} anymore. 

In this paper, we use lowercase German letters for the Lie algebras of Lie groups denoted by uppercase Roman letters. 
For example, the Lie algebras of $G$, $K_H$, and $\Stab(X)$ are $\g$, $\k_H$, and $\stab(X)$, respectively. 
Then, our main result is stated as follows: 

\begin{thm}\label{thm:main}
Let $G$ be a Lie group, $H$ its closed subgroup with finitely many connected components, and $N$ the codimension of $H$ in $G$. 

\textup{(1)} If $(\Lambda^N (\g/\h)^\ast)^\h \neq 0$ and $H^N(\g, \h; \R) = 0$, 
then there is no compact manifold locally modelled on $G/H$. 

\textup{(2)} Take a maximal compact subgroup $K_H$ of $H$. 
Let 
\[
i: H^N(\g, \h; \R) \to H^N(\g, \k_H; \R)
\]
be the homomorphism induced by the inclusion of Lie algebras $\k_H \subset \h$. 
If $i$ is not injective, 
then there is no compact manifold locally modelled on $G/H$. 
\end{thm}

Some applications of this theorem are given in Sections 6--7. 

The idea of Theorem~\ref{thm:main} (1) is already implicit in \cite{Ben-Lab92}. We shall give its proof for the sake of completeness. 
Theorem~\ref{thm:main} (2) is proved in \cite{Mor} 
under the assumptions that $G$ is reductive and $H$ is reductive in $G$. 
Our improvement is to separate the Poincar\'e duaity argument from the other parts of the proof (cf. Proposition~\ref{prop:lower-deg}). 
This enables us to prove the theorem in general situations. 

Theorem~\ref{thm:main} generalizes some earlier results in \cite{Kob-Ono90}, \cite{Kob89}, \cite{Ben-Lab92}, and \cite{Mor} (see Section 5). 

\section{Preliminaries}
In this section, we review the definition of a homomorphism $\eta: H^p(\g, H; \R) \to H^p(M; \R)$, which plays a foundational role in the cohomological study of Problem~\ref{prob:cpt}. 

Let $X$ be a real analytic manifold with an action of a Lie group $G$. 
Recall that a $(G, X)$-structure on a manifold $M$ is a collection of 
$(U_i)_{i \in I}$, $(\phi_i)_{i \in I}$, $(g_{ij})_{i,j \in I}$, where $(U_i)_{i \in I}$ is an open covering of $M$, 
$\phi_i$ is a diffeomorphism from $U_i$ to some open set of $X$, 
and 
$g_{ij}: U_i \cap U_j \to G$ is a locally constant map satisfying
\[
g_{ij}(p)  \phi_j(p) = \phi_i(p) \quad (p \in U_i \cap U_j). 
\]
We assume the cocycle condition for the transition functions $(g_{ij})_{i,j \in I}$: 
\[
g_{ii}(p) = 1 \quad (p \in U_i), \qquad
g_{ij}(p)  g_{jk}(p)  g_{ki}(p) = 1 \quad (p \in U_i \cap U_j \cap U_k).
\]
It is automatically satisfied if $X$ is connected and $G$ acts on $X$ effectively. 
We mainly consider the case when $G$ acts transitively on $X$, 
namely, $X = G/H$ for some closed subgroup $H$ of $G$. 
A manifold equipped with a $(G, G/H)$-structure is also called 
a manifold locally modelled on $G/H$. 

Let $M$ be a manifold equipped with a $(G, X)$-structure 
$(U_i)_{i \in I}$, $(\phi_i)_{i \in I}$, $(g_{ij})_{i,j \in I}$. 
Let $\pi : E \to X$ be a $G$-equivariant fibre bundle on $X$ with typical fibre $F$. 
Patching $(\phi_i^\ast E)_{i \in I}$ by $(g_{ij})_{i,j \in I}$, 
we get a fibre bundle $\pi_M : E_M \to M$ with the same typical fibre $F$. 
We call it the locally $G$-equivariant bundle over $M$ corresponding to $E$. 
By definition $E_M$ naturally equips a $(G, E)$-structure. 
We can define 
\[
\eta: \Gamma(X; E)^G \to \Gamma(M; E_M)
\]
also by patching construction. 
In particular, if $X=G/H$ and $E = \Lambda^p T^\ast X$, this is written as 
\[
\eta : (\Lambda^p(\g/\h)^\ast)^H \to \Omega^p(M). 
\]
Here, we naturally identified 
$\Omega^p(G/H)^G$ with $(\Lambda^p(\g/\h)^\ast)^H$. 
Taking cohomology, we get a homomorphism
\[
\eta : H^p(\g, H; \R) \to H^p(M; \R) 
\]
(see e.g. \cite[\S 1.3]{Fuk}, \cite[\S 2.2]{SMor} for the definition of relative Lie algebra cohomology $H^p(\g, H; \R)$). 
Such a homomorphism $\eta$ appears explicitly or implicitly in many branches of geometry and representation theory, 
e.g. the Matsushima--Murakami formula \cite{Mat-Mur63}, 
characteristic classes of foliations \cite{Bot-Hae72}, 
a generalization of Hirzebruch's proportionality principle \cite{Kob-Ono90}, 
and the existence problem of a compact manifold locally modelled on homogeneous spaces \cite{Kob-Ono90}, \cite{Ben-Lab92}, \cite{Mor}. 

\section[Proof of Theorem 1.2]{Proof of Theorem~\ref{thm:main}}

\begin{lem}\label{lem:H_0}
Let $G$ be a Lie group and $H$ its closed subgroup with finitely many connected components. We write $H_0$ for the identity component of $H$. 
If there is no compact manifold locally modelled on $G/H_0$, neither is on $G/H$. 
\end{lem}

\begin{proof}[Proof of Lemma~\ref{lem:H_0}]
This is well-known at least for Clifford--Klein forms. Suppose there is a compact manifold $M$ locally modelled on $G/H$. Consider the locally $G$-equivariant fibre bundle $\pi_M: M_0 \to M$ corresponding to $\pi : G/H_0 \to G/H$. Then the total space $M_0$ is locally modelled on $G/H_0$ and compact. 
\end{proof}

Thus we may assume $H$ to be connected without loss of generality. 
Now, it is enough to see: 

\begin{prop}\label{prop:gH}
Let $G$ be a Lie group, $H$ its closed subgroup, and $N$ the codimension of $H$ in $G$. 

\textup{(1)} If $(\Lambda^N (\g/\h)^\ast)^H \neq 0$ and $H^N(\g, H; \R) = 0$, then there is no compact manifold locally modelled on $G/H$. 

\textup{(2)} Suppose that $H$ has finitely many connected components. 
Take a maximal compact subgroup $K_H$ of $H$. If the homomorphism 
\[
i: H^N(\g, H; \R) \to H^N(\g, K_H; \R)
\]
is not injective, then there is no compact manifold locally modelled on $G/H$. 
\end{prop}
\begin{rem}
Proposition~\ref{prop:gH} (1) holds true even if $H$ has infinitely many connected components. 
\end{rem}
\begin{proof}[Proof of Proposition~\ref{prop:gH}]
(1) Suppose, on the contrary, that there is a compact manifold $M$ locally modelled on $G/H$. 
Take a nonzero element $\Phi$ of $(\Lambda^N (\g/\h)^\ast)^H$; 
it is identified with a $G$-invariant volume form on $G/H$. 
Hence $\eta(\Phi) \in \Omega^N(M)$ is a volume form on $M$ by construction of $\eta$, 
and $[\eta(\Phi)] \neq 0$ in $H^N(M; \R)$ by compactness of $M$. 
On the other hand, $[\Phi] = 0$ in $H^N(\g, H; \R)$ by assumption, 
and $[\eta(\Phi)] = 0$ in $H^N(M; \R)$. This is contradiction. 

(2) Let $M$ be a compact manifold locally modelled on $G/H$. 
Let $\pi_M: E_M \to M$ be the locally $G$-equivariant fibre bundle on $M$ 
corresponding to $\pi: G/K_H \to G/H$. 
Consider the following commutative diagram: 
\[
 \begin{CD}
  H^N(\g, H; \R) @>{i}>> H^N(\g, K_H; \R) \\
  @V{\eta}VV  @V{\eta}VV \\
  H^N(M; \R) @>{\pi_M^\ast}>> H^N(E_M; \R). \\
 \end{CD}
\]
We saw in the proof of (1) that the homomorphism $\eta: H^N(\g, H; \R) \to H^N(M; \R)$ is injective. 
The typical fibre $H/K_H$ of the fibre bundle $\pi_M: E_M \to M$ 
is contractible by the Cartan--Malcev--Iwasawa--Mostow theorem (cf. \cite[Ch.\! XV, Th.\! 3.1]{Hoc65}), 
thus $\pi_M^\ast : H^N(M; \R) \to H^N(E_M; \R)$ is an isomorphism. 
These yield the injectivity of $i: H^N(\g, H; \R) \to H^N(\g, K_H; \R)$. 
\end{proof}

\section[Equivalent form of Theorem 1.2 (1)]{Equivalent form of Theorem~\ref{thm:main} (1)}

It is sometimes useful to rewrite Theorem~\ref{thm:main} (1) as follows: 

\begin{prop}\label{prop:trace-free}
Let $G$ be a Lie group and $H$ its closed subgroup with finitely many connected components. 
Let $\n_\g(\h)$ denote the normalizer of $\h$ in $\g$. 
If the $\h$-action on $\g/\h$ is trace-free 
(i.e. $\tr(\ad_{\g/\h}(X)) = 0$ for all $X \in \h$) 
and the $\n_\g(\h)$-action on $\g/\h$ is not trace-free, 
then there is no compact manifold locally modelled on $G/H$. 
\end{prop}

\begin{proof}
This is a direct consequence of Theorem~\ref{thm:main} (1) and the lemma below. 
\end{proof}

\begin{lem}\label{lem:top-term}
Let $\g$ be a Lie algebra, $\h$ its subalgebra, and $N$ the codimension of $\h$ in $\g$. 

\textup{(1)} The $\h$-action on $\g/\h$ is trace-free if and only if $(\Lambda^N (\g/\h)^\ast)^\h \neq 0$.

\textup{(2)} The $\n_\g(\h)$-action on $\g/\h$ is trace-free if and only if $H^N(\g, \h; \R) \neq 0$. 
\end{lem}

\begin{proof}
(1) This follows immediately from the definition of a $\h$-action on $\Lambda^N (\g/\h)^\ast$. 

(2) Let $\iota$ denote the interior product and $\L$ the $\g$-action on $\Lambda \g^\ast$. 
Assume that $(\Lambda^N (\g/\h)^\ast)^\h \neq 0$ 
and fix a nonzero element $\Phi$ of $(\Lambda^N (\g/\h)^\ast)^\h$. 
We wish to determine when 
\[
d: (\Lambda^{N-1} (\g/\h)^\ast)^\h \to (\Lambda^N (\g/\h)^\ast)^\h
\]
is a zero map. 
Every element of $\Lambda^{N-1} (\g/\h)^\ast$ is written in the form 
$\iota(Y) \Phi$ ($Y \in \g$) and the choice of such $Y$ is unique up to $\h$. 
For $X \in \h$, 
\[
\L(X) \iota(Y) \Phi = \iota(Y) \L(X) \Phi - \iota([X,Y]) \Phi = \iota([X,Y])\Phi.
\]
It is equal to zero if and only if $[X, Y] \in \h$. 
Thus $\iota(Y) \Phi$ is $\h$-invariant if and only if $Y \in \n_\g(\h)$. 
Now, 
\[
d\iota(Y)\Phi = \L(Y) \Phi - \iota(Y)d\Phi = \L(Y) \Phi = -\tr(\ad_{\g/\h}(Y)) \Phi.
\]
Hence $d = 0$ on $(\Lambda^{N-1} (\g/\h)^\ast)^\h$ 
if and only if the $\n_\g(\h)$-action on $\g/\h$ is trace-free.
\end{proof}

\section{Relation with earlier results}

Kobayashi and Ono established necessary conditions for the existence of compact Clifford--Klein forms (\cite[Cor.\! 5]{Kob-Ono90}, \cite[Prop.\! 4.10]{Kob89}) 
using a cohomological method. 
We gave a generalization \cite[Th.\! 1.3]{Mor} of their necessary conditions. 
The following proposition shows that Theorem~\ref{thm:main} (2) further generalizes \cite[Th.\! 1.3]{Mor}. 

\begin{prop}\label{prop:lower-deg}
Let $G$ be a unimodular Lie group, $H$ its closed subgroup such that $\h$ is reductive in $\g$, and $N$ the codimension of $H$ in $G$. 
If $i: H^p(\g, \h; \R) \to H^p(\g, \k_H; \R)$ is injective for $p=N$, 
it is also injective for $0 \leq p \leq N-1$. 
\end{prop}
\begin{rem}
In this paper, we say that a Lie group $G$ is unimodular if the adjoint action of $\g$ on itself is trace-free. 
If $G$ is connected, it is equivalent to the existence of bi-invaraint Haar measure on $G$. 
\end{rem}
\begin{proof}[Proof of Proposition~\ref{prop:lower-deg}]
This follows from the standard Poincar\'e duality argument. 
Take any nonzero cohomology class $\alpha \in H^p(\g, \h; \R)$. 
By the Poincar\'e duality \cite[Th.\! 12.1]{Kosz50}, 
we can pick $\beta \in H^{N-p}(\g, \h; \R)$ such that 
$\alpha \w \beta \neq 0$ in $H^N(\g, \h; \R)$. 
Then $\eta(\alpha \w \beta) \neq 0$ by assumption, 
which yields $\eta(\alpha) \neq 0$. 
\end{proof}

We can also recover a result of Benoist--Labourie \cite{Ben-Lab92} from Theorem~\ref{thm:main}, 
though our proof relies on the crucial parts of \cite{Ben-Lab92}.

\begin{prop}[{\cite[Th.\! 1]{Ben-Lab92}}]\label{prop:Ben-Lab92} 
Let $G$ be a connected semisimple Lie group and $H$ its unimodular subgroup with finitely many connected components. 
If the centre $\z(\h)$ of $\h$ contains a nonzero hyperbolic element, 
then there is no compact manifold locally modelled on $G/H$. 
\end{prop}

\begin{proof}
We may assume $H$ to be connected by Lemma~\ref{lem:H_0}. 
We identify $\g$ with $\g^\ast$ via the Killing form. 
In \cite{Ben-Lab92}, it is shown that our assumptions yield the existence of $X \in \g$ such that: 
\begin{itemize}
\item $X$ is a nonzero hyperbolic element. 
\item $H \subset \Stab(X)$. 
\item Let $\omega = dX$. Let $N$ and $2m$ be the codimensions of $H$ and $\Stab(X)$ in $G$, respectively. If we take $\mu \in (\Lambda^{N-2m}(\g/\h)^\ast)^\h$ so that $\mu \w \omega^m \neq 0$, then $d(\mu \w \omega^{m-1}) = 0$. 
\end{itemize}
Here, $\Stab(X) \subset G$ is the stabilizer of $X$ in $G$. 
Remark that $\omega = dX$ is an element of $(\Lambda^2(\g/\stab(X))^\ast)^{\Stab(X)} \ (\subset (\Lambda^2(\g/\h)^\ast)^\h)$ and satisfies $\omega^m \neq 0$ ($2m = \dim (G/\Stab(X))$). 

If $[\mu \w \omega^m]_{\g, \h} = 0$ in $H^N(\g, \h; \R)$, then the proposition follows from Theorem~\ref{thm:main} (1). 
Thus we assume $[\mu \w \omega^m]_{\g, \h} \neq 0$. 
Since every element of $\k_H$ commutes with $X$ and is elliptic, 
$X \in ((\g/\k_H)^\ast)^{\k_H}$. 
Hence 
$[\mu \w \omega^m]_{\g, \k_H} = [d(X \w \mu \w \omega^{m-1})]_{\g, \k_H} = 0$ 
in $H^N(\g, \k_H; \R)$. Apply Theorem~\ref{thm:main} (2). 
\end{proof}

\section{Examples (1): nonreductive Lie groups}

In the rest of this paper, we shall give some applications of Theorem~\ref{thm:main}. 
In this section, we study the case that $G$ is nonreductive. 

\begin{ex}\label{ex:levi}
Let $G$ be a simply connected nonunimodular Lie group and
\[
G = S \ltimes R \quad (\text{$S$: semisimple, $R$: solvable})
\]
be its Levi decomposition. 
Take any closed unimodular subgroup $H$ of $S$ with finitely many connected components. 
Then there is no compact manifold locally modelled on $G/H$. 
\end{ex}

In fact, we can show a slightly more general result: 

\begin{ex}\label{ex:nonunim}
Let $G$ be a nonunimodular Lie group. Let $G'$ be a closed subgroup of $G$ such that $\g'$ is reductive in $\g$ and the adjoint action of $\z(\g')$ on $\g$ is trace-free. Here $\z(\g')$ denotes the centre of $\g'$. 
Let $H$ be any closed unimodular subgroup of $G'$ with finitely many connected components. 
Then there is no compact manifold locally modelled on $G/H$. 
\end{ex}

\begin{proof}[Proof of Example~\ref{ex:nonunim}]
By Proposition~\ref{prop:trace-free}, it suffices to check that: 
\begin{enumerate}[(i)]
\item The $\h$-action on $\g/\h$ is trace-free. 
\item The $\n_\g(\h)$-action on $\g/\h$ is not trace-free. 
\end{enumerate}
We will show the stronger results: 
\begin{enumerate}
\item[(i$'$)] The $\g'$-action on $\g$ is trace-free. 
\item[(ii$'$)] The $\z_\g(\g')$-action on $\g$ is not trace-free. 
\end{enumerate}
Here $\z_\g(\g')$ denotes the centralizer of $\g'$ in $\g$. 

Let us prove (i$'$). Since $\g'$ is reductive, 
we have a direct sum decomposition $\g' = \z(\g') \oplus [\g', \g']$. 
By our assumption, $\z(\g')$ acts trace-freely on $\g$. 
Also, $[\g', \g']$ acts trace-freely on $\g$ since it is a semisimple Lie algebra.

Now let us prove (ii$'$). Let 
\[
\g_1 = \{ X \in \g : \tr(\ad_\g(X)) = 0 \}. 
\]
Since $\g'$ is reductive in $\g$, 
we can pick a $\g'$-invariant subspace $\g_2$ complementary to $\g_1$ in $\g$. 
Note that $\g_2 \neq \{ 0 \}$ and $\tr(\ad_\g(X)) \neq 0$ for any nonzero element $X$ of $\g_2$.
We have $[\g', \g_2] \subset [\g, \g] \subset \g_1$, while $[\g', \g_2] \subset \g_2$ by $\g'$-invariance of $\g_2$. This means $\g_2 \subset \z_\g(\g')$. From these (ii$'$) follows. 
\end{proof}

Next we consider coadjoint orbits. Let $G$ be a Lie group and $F \in \g^\ast$. The coadjoint orbit $G. F \subset \g^\ast$ of $F$ is $G$-diffeomorphic to $G/\Stab(F)$, where $\Stab(F) = \{ g \in G : g.F = F \}$ is the stabilizer of $F$ in $G$. 
Let $\omega = dF$, in other words, 
\[
\omega (X, Y) = - \langle F, [X,Y] \rangle \quad (X, Y \in \g). 
\]
Then $\omega$ is an element of $(\Lambda^2 (\g/\stab(F))^\ast)^{\Stab(F)}$ satisfying $d\omega = 0$ and $\omega^m \neq 0$ ($2m = \dim(G/\Stab(F))$). Under the identification $(\Lambda^2 (\g/\stab(F))^\ast)^{\Stab(F)} \simeq \Omega^2(G/\Stab(F))^G$, $\omega$ corresponds to the Kirillov--Kostant--Souriau symplectic form. 
Applying Theorem~\ref{thm:main} to this setting, we obtain:

\begin{ex}\label{ex:coad}
Let $G$ be a Lie group and $F \in \g^\ast$. 
Assume that $\dim (G/\Stab(F)) > 0$ and $\Stab(F)$ has finitely many connected components.
If $F|_{\k_{\Stab(F)} \cap [\g,\g]} = 0$, 
then there is no compact manifold locally modelled on $G/\Stab(F)$.
\end{ex}
\begin{rem}
The condition $\dim (G/\Stab(F)) > 0$ holds if and only if $F|_{[\g, \g]} \neq 0$. 
\end{rem}
\begin{rem}
If $G$ is a real linear algebraic group, the number of the connected components of $\Stab(F)$ (in the Euclidean topology) is always finite by Whitney's theorem \cite[Th.\! 3]{Whi57}. 
For a nonalgebraic Lie group $G$, it may be infinite. An easy example is: 
\[
G = (\text{universal covering of $SL(2, \R)$}), \quad  F = \begin{pmatrix} 0 & 0 \\ 1 & 0 \end{pmatrix} \in \g \simeq \g^\ast.
\]
Here we identified $\g$ with $\g^\ast$ via the Killing form. 
\end{rem}
\begin{proof}[Proof of Example~\ref{ex:coad}]
Put $2m = \dim(G/\Stab(F))$. 
Recall that $\omega^m$ is a nonzero element of 
$(\Lambda^{2m} (\g/\stab(F))^\ast)^{\stab(F)}$. 
By Theorem~\ref{thm:main} (1), we only need to consider the case that $[\omega^m]_{\g, \stab(F)} \neq 0$. 
Thus, by Theorem~\ref{thm:main} (2), it suffices to prove that $[\omega^m]_{\g, \k_{\Stab(F)}} = 0$. 
Since 
\[
\ker (d: \g^\ast \to \Lambda^2 \g^\ast) = (\g^\ast)^\g = (\g/[\g,\g])^\ast,
\]
our assumption $F|_{\k_{\Stab(F)} \cap [\g,\g]} = 0$ may be rewritten as:
\[
F+F' \in ((\g/\k_{\Stab(F)})^\ast)^{\k_{\Stab(F)}} \quad \text{for some} \ F' \in \ker (d: \g^\ast \to \Lambda^2 \g^\ast). 
\]
We obtain 
\[
[\omega^m]_{\g, \k_{\Stab(F)}} = [d((F+F') \w \omega^{m-1})]_{\g, \k_{\Stab(F)}} = 0 \quad \text{in} \ H^{2m}(\g, \k_{\Stab(F)}; \R) 
\]
as required. 
\end{proof}

When $G$ is a linear solvable Lie group, Example~\ref{ex:coad} gives the following result: 

\begin{ex}\label{ex:solvorbit}
Let $G$ be a linear solvable Lie group and $F \in \g^\ast$. 
Assume that $\dim (G/\Stab(F)) > 0$ and $\Stab(F)$ has finitely many connected components. 
Then there is no compact manifold locally modelled on $G/\Stab(F)$. 
\end{ex}
\begin{rem}
In Example~\ref{ex:solvorbit}, if $G$ is simply connected, then $G/\Stab(F)$ admits an infinite discontinuous group (\cite[Th.\! 2.2]{Kob93}). 
\end{rem}
\begin{rem}
In Example~\ref{ex:solvorbit}, the linearity of $G$ is crucial. 
Consider the nonlinear nilpotent Lie group 
\[
G := \left\{ \begin{pmatrix}1 & a & c \\ & 1 & b \\ & & 1\end{pmatrix} : a,b,c \in \R \right\} / \left\{ \begin{pmatrix}1 & 0 & n \\ & 1 & 0 \\ & & 1\end{pmatrix} : n \in \Z \right\}. 
\]
Its 2-dimensional coadjoint orbits have connected stabilizers, 
but admit compact Clifford--Klein forms. 
\end{rem}
\begin{proof}[Proof of Example~\ref{ex:solvorbit}]
Let $G_0$ be the identity component of $G$ and $[G_0, G_0]$ be its commutator subgroup. 
Then $[G_0, G_0]$ is closed in $G$ and it does not contain a compact subgroup other than $\{ 1 \}$ \cite[Ch.\! XVIII, Th.\! 3.2]{Hoc65}. 
In particular $K_{\Stab(F)} \cap [G_0, G_0] = \{ 1 \}$ and hence $\k_{\Stab(F)} \cap [\g, \g] = 0$. 
Thus, we can apply Example~\ref{ex:coad}. 
\end{proof}

\section{Examples (2): reductive Lie groups}

In this section, we study the case that $G$ is reductive and $H$ is not reductive in $G$. 
Note that, when $G$ is reductive and $H$ is reductive in $G$, 
Theorem~\ref{thm:main} (1) is not applicable and, as we saw in Section 5, Theorem~\ref{thm:main} (2) is identical to \cite[Th.\! 1.3]{Mor}. 

\begin{ex}\label{ex:parab}
Let $G$ be a reductive Lie group and $P = MAN$ be a proper parabolic subgroup of $G$. Then there is no compact manifold locally modelled on $G/N$. 
\end{ex}
\begin{proof}
Since $\g$ and $\n$ are unimodular, the $\n$-action on $\g/\n$ is trace-free. 
On the other hand, $\a$ normalizes $\n$ and contains an element $X$ such that $\tr_\n(X) \neq 0$. Since $\g$ is unimodular, such $X$ also satisfies $\tr_{\g/\n}(X) \neq 0$. Thus, we can apply Proposition~\ref{prop:trace-free}.
\end{proof}

\begin{ex}\label{ex:ssorbit}
Let $G$ be a real linear semisimple algebraic group and $X \in \g$. 
Let $\Stab(X) \subset G$ be the stabilizer of $X$ in $G$. 
Let $X=X_e + X_h + X_n$ be the decomposition of $X$ into elliptic, hyperbolic, and nilpotent parts. 
If $X$ is not a semisimple element (i.e. $X_n \neq 0$), 
then there is no compact manifold locally modelled on $G/\Stab(X)$. 
\end{ex}
\begin{rem}
The study of Problem~\ref{prob:cpt} for $G/\Stab(X)$, 
where $G$ and $X$ are as in Example~\ref{ex:ssorbit}, 
was started by \cite{Kob92}, and then extended by \cite{Ben-Lab92}. 
We list their results here: 
\begin{itemize}
\item Assume that $X$ is a semisimple element (i.e. $X_n = 0$). 
If $\Stab(X) \neq \Stab(X_e)$, namely, 
if $G/\Stab(X)$ does not carry a $G$-invariant complex structure, 
then $G/\Stab(X)$ does not admit a compact Clifford--Klein form. 
only if 
(\cite[Th.\! 1.3]{Kob92}).
\item If $X$ is a nilpotent element (i.e. $X = X_n$), 
then there is no compact manifold locally modelled on $G/\Stab(X)$ 
(\cite[Cor.\! 4]{Ben-Lab92}). 
\item If $X_h \neq 0$, 
then there is no compact manifold locally modelled on $G/\Stab(X)$ 
(\cite[Cor.\! 5]{Ben-Lab92}).
\end{itemize}
Combining \cite[Cor.\! 5]{Ben-Lab92} and Example~\ref{ex:ssorbit}, we conclude that, if $X$ is not an elliptic element (i.e. if $X \neq X_e$), then there is no compact manifold locally modelled on $G/\Stab(X)$. 
\end{rem}
\begin{proof}[Proof of Example~\ref{ex:ssorbit}]
We identify $\g$ with $\g^\ast$ via the Killing form. 
Let $\omega = dX$. 
Then $\omega$ is an element of $(\Lambda^2(\g/\stab(X))^\ast)^{\Stab(X)}$ 
satisfying $d\omega = 0$ and $\omega^m \neq 0$ ($2m = \dim (G/\Stab(X))$). 
By Theorem~\ref{thm:main} (1), we may assume $[\omega^m]_{\g, \stab(X)} = 0$. 
Then, by Theorem~\ref{thm:main} (2), it is enough to prove that $[\omega^m]_{\g, \k_{\Stab(X)}} = 0$. 

Put $X_{ss} = X_e + X_h$. 
Let $\omega_{ss} = dX_{ss}$ and $\omega_n = dX_n$. 
They are elements of $(\Lambda^2 (\g/\stab(X))^\ast)^{\stab(X)}$ 
because $Y \in \g$ commutes with $X$ if and only if it commutes with $X_{ss}$ and $X_n$. 
Since every element of $\k_{\Stab(X)}$ commutes with $X_n$ and is elliptic, $X_n$ is perpendicular to $\k_{\Stab(X)}$. 
Therefore, $X_n \in ((\g/\k_{\Stab(X)})^\ast)^{\k_{\Stab(X)}}$. We have 
\begin{align*}
[\omega^m]_{\g,\k_{\Stab(X)}} &= [\sum_{k=0}^m \frac{m!}{k!(m-k)!} \, \omega_{ss}^{m-k} \w \omega_n^{k}]_{\g,\k_{\Stab(X)}} \\ 
&= [\omega_{ss}^m + d ( X_n \w \sum_{k=1}^m \frac{m!}{k!(m-k)!} \, \omega_{ss}^{m-k} \w \omega_n^{k-1} )]_{\g,\k_{\Stab(X)}} \\ 
&= [\omega_{ss}^m]_{\g,\k_{\Stab(X)}} 
\quad \text{in $H^{2m}(\g, \k_{\Stab(X)}; \R)$}. 
\end{align*}
Let us prove that $\omega_{ss}^m = 0$. 
To see this, it suffices to show that $\stab(X) \subsetneq \stab(X_{ss})$. 
Let us assume the contrary: $\stab(X) = \stab(X_{ss})$. 
Take a Cartan subalgebra $\j$ of $\g \otimes \C$ containing $X_{ss}$. 
Then we have 
\[
\j \subset \stab(X_{ss}) \otimes \C = \stab(X) \otimes \C \subset \stab(X_n) \otimes \C. 
\]
Since $\j$ is a maximal abelian subalgebra of $\g \otimes \C$, we have $X_n \in \j$. 
This is impossible because $\j$ consists of semisimple elements. 
\end{proof}

\subsection*{Acknowledgements}
The author is very grateful to Professor Toshiyuki Kobayashi for his warm encouragement and many valuable suggestions. 
This work was supported by JSPS KAKENHI Grant Number 14J08233 and the Program for Leading Graduate Schools, MEXT, Japan.

\noindent
\textsc{Graduate School of Mathematical Science, \\ The University of Tokyo, \\
3-8-1 Komaba, Meguro-ku, Tokyo 153-8914, Japan} \\
\textit{E-mail address}: \texttt{ymorita@ms.u-tokyo.ac.jp}

\end{document}